\documentclass[11pt]{article}
\usepackage[round]{natbib}
\usepackage[T1]{fontenc}
\usepackage[utf8]{inputenc}
\usepackage{amsmath}
\usepackage{enumitem}
\usepackage{amsthm}
\usepackage[left=1 in,right=1 in,top=1 in,bottom=1 in]{geometry}
\usepackage{amssymb}
\usepackage{adjustbox}
\usepackage[title]{appendix}
\usepackage{mwe}
\usepackage{bbold}
\usepackage{algorithm}
\usepackage{algpseudocode}
\usepackage{quiver}
\usepackage{textcomp}
\usepackage{siunitx}
\usepackage{xfrac}
\usepackage{bbm}
\usepackage{amsthm}
\usepackage[colorlinks=true]{hyperref}
\usepackage{graphicx}
\usepackage{float}
\usepackage{subcaption}
\usepackage{setspace}
\usepackage{tabto}
\usepackage[nottoc, notlof, notlot]{tocbibind}
\theoremstyle{definition}
\newtheorem{defi}{Definition}
\theoremstyle{plain}
\newtheorem{thm}{Theorem}
\newtheorem{crl}{Corollary}
\newtheorem{lmm}{Lemma}
\theoremstyle{remark}

\theoremstyle{plain}
\newtheorem{prp}{Proposition}
\theoremstyle{plain}

\theoremstyle{plain}

\newtheorem{innercustomthm}{Theorem}
\newenvironment{customthm}[1]
  {\renewcommand\theinnercustomthm{#1}\innercustomthm}
  {\endinnercustomthm}

\allowdisplaybreaks

\usepackage{bigfoot}

\DeclareNewFootnote{AAffil}[arabic]
\DeclareNewFootnote{ANote}[fnsymbol]

\usepackage{etoolbox}
\makeatletter
\patchcmd\maketitle{\def\@makefnmark{\rlap{\@textsuperscript{\normalfont\@thefnmark}}}}{}{}{}
\makeatother

\makeatletter
\def\thanksAAffil#1{% <--- These %'s are necessary for spacing
  \footnotemarkAAffil\protected@xdef\@thanks{\@thanks%
        \protect\footnotetextAAffil[\the \c@footnoteAAffil]{#1}}%
}
\def\thanksANote#1{%
  \footnotemarkANote%
  \protected@xdef\@thanks{\@thanks%
        \protect\footnotetextANote[\the \c@footnoteANote]{#1}}%
}
\makeatother
\begin{document}
\title{Persistence Diagram Estimation : Beyond Plug-in Approaches}
\author{Hugo Henneuse \thanksAAffil{Laboratoire de Mathématiques d'Orsay, Université Paris-Saclay, Orsay, France}$^{\text{ ,}}$\thanksAAffil{DataShape, Inria Saclay, Palaiseau, France}\\\href{mailto:hugo.henneuse@universite-paris-saclay.fr}{hugo.henneuse@universite-paris-saclay.fr}}
\maketitle
\begin{abstract}
\noindent Persistent homology is a tool from Topological Data Analysis (TDA) used to summarize the topology underlying data. It can be conveniently represented through persistence diagrams. Observing a noisy signal, common strategies to infer its persistence diagram involve plug-in estimators, and convergence properties are then derived from sup-norm stability. This dependence on the sup-norm convergence of the preliminary estimator is restrictive, as it essentially imposes to consider regular classes of signals. Departing from these approaches, we design an estimator based on image persistence. In the context of the Gaussian white noise model, and for large classes of piecewise-constant signals, we prove that the proposed estimator is consistent and achieves parametric rates.
\end{abstract}
\section*{Introduction}
Topological Data Analysis (TDA) is a field that aims to provide representations to describe the ``shape'' of data. A central tool in TDA is persistent homology and its representation through persistence diagrams. Persistent homology permits to encode the evolution of topological features (in the homology sense) along a family of nested spaces, called filtration. Moving along indices, topological features (connected components, cycles, cavities, ...) can appear or die (existing connected components merge, cycles or cavities are filled, ...). Persistence diagrams then offer a practical, multiscale, summary of the topology underlying data. Although in its early years research on persistent homology has primarily focused on deterministic settings, it didn't take long for statistical questions to emerge. A prominent topic is the estimation of persistence diagram and, more specifically, the investigation of convergence properties of considered estimators.\\
A first historical success in this direction is \cite{BubenikKim07}, that formalizes the problem of persistence diagram estimation from the (sub or super) level sets of a density, in several parametric settings. Notably, when observing points sampled from a Mises-Fischer distribution with fixed parameter $\kappa$ on $S^{p-1}$, they show that a plug-in estimator from the maximum-likelihood estimation of $\kappa$ permits to recover consistently the persistence diagram of the density.\\
 \cite{Fasy} consider the problem of estimating the persistent homology of the support $M$ of a distribution $P$ (associated to its offset filtration $(B_{2}(M,r))_{r\geq 0}$) while observing sampled points from $P$. Under geometric control over both $M$ and $P$ they show that the persistence diagram associated to the offsets of the sampled points provides a consistent estimator. They additionally derive confidence sets, showing that it suffices to remove points lying on a small band around the diagonal of the persistence diagram to separate topological noise from true topological features. Their analysis exploits (a variant of) the sup-norm stability theorem \citep{Baranikov94,CSEH2005,Chazal2009}, showing that, the set of sampled points converges to $M$ in Hausdorff distance. On the same problem, \cite{ChazalGlisseMichel} provide consistent estimators and show that they achieve (nearly) minimax rates under more general assumptions. This work also exploits the sup-norm stability.\\
 In the context of the non-parametric regression, closer to the setting considered in this work, \cite{BCL2009} study a plug-in estimator based on kernel estimation for the persistence diagram associated to sublevel sets of the regression function. Exploiting the sup-norm stability again, they prove the consistency of this estimator and that it achieves minimax rates on Hölder spaces. These rates coincide with the known minimax rates for the estimation of the regression function in sup-norm. It essentially means that the estimation of the persistence diagram from sublevel sets is as difficult as estimating the regression function in sup norm over Hölder spaces.\\
 As highlighted, a common aspect of all the previously cited works is that they adopt plug-in approaches and convergence results are established through the sup-norm stability theorem (or variants of the sup-norm stability theorem). This dependence on the convergence of the preliminary estimators is limiting as it essentially imposes to consider regular objects (signal, regression function, density, ...). This limitation motivates investigations of persistent homology inference over wider classes, typically where consistent estimation of such objects in sup-norm is no longer possible. A difficulty is that it requires to departs from approaches that rely on sup-norm stability.\\
 A first step in this direction is proposed by \cite{Bobrowski}. For the non-parametric regression and the density model, they propose a non-plug-in approach to infer the persistence diagrams coming from the superlevel sets filtration of a function $f$, respectively, the regression function and the density. They consider, for all, $L\in\mathbb{R}$, $\widehat{\mathcal{D}}_{L}$ an estimator of the superlevel set $f^{-1}([L,+\infty[)$ based on kernels. Instead of directly estimating the underlying persistent modules by taking the homology groups of the filtration $(\widehat{\mathcal{D}}_{L})_{L\in\mathbb{R}}$, they consider, for all $s\in\mathbb{N}$ and $\lambda\in\mathbb{R}$, the map :
 $$i_{s,L}:H_{s}\left(\widehat{\mathcal{D}}_{L+\epsilon}\right)\rightarrow \widehat{\mathcal{D}}_{L-\epsilon}$$
 induced by the inclusion $\widehat{\mathcal{D}}_{L+\epsilon}\subset \widehat{\mathcal{D}}_{L-\epsilon}$, with $\epsilon>0$ a chosen precision parameter. And then consider the modules $(\operatorname{Im}(i_{s,L}))_{L\in\mathbb{R}}$. This additional step of constructing an image persistent module can be thought as a topological regularization, which allows considering wide classes of signal. In particular, supposing that $f$ is bounded, and the module associated to its superlevel sets filtration is $q-$tame (see definition \ref{def: q-tame} in Section \ref{Background section}), they show that the bottleneck distance (see definition \ref{def: bottleneck} in Section \ref{Background section}) between the diagram induced by the proposed image persistent module and the true persistence diagram is bounded by $5\varepsilon$ with high probability (depending also on $\varepsilon$). Their analysis relies on a weaker notion of stability : the algebraic stability \citep{Chazal2009}. Yet, in this framework the optimal calibration of $\epsilon$ is not obvious, and thus this work does not permit to establish convergence rates or prove proper consistency. The difficulty stems from their framework being too wide, allowing for wildly irregular functions. Typically, it is easy to construct two functions satisfying the $q-$tameness and boundedness assumptions, that only differ on a null set (and thus statistically indistinguishable), that have (arbitrarily) far persistence diagrams in bottleneck distance.\\
 Hence, it is interesting to identify narrower classes on which we can provide stronger and more precise results. Recently, in \cite{Henneuse24a}, we proposed a new approach that also departs from the sup norm stability. This work studies the inference of persistence diagram, from a minimax perspective, for the Gaussian white noise model and the non-parametric regression. Although we consider a plug-in estimator, our analysis of convergence properties did not rely on sup-norm stability but on algebraic stability. We introduced and studied classes of piecewise-Hölder continuous functions with discontinuities set having a positive reach. The reach \citep{Fed59} is a popular curvature measure in geometric inference, that can be thought as a way to describe the geometric regularity of a set. Under this reach assumption, we manage to show that a histogram plug-in estimator permits to achieve the well known minimax rates for Hölder-continuous functions. Still, the positive reach assumption imposes some limitations on the shape of discontinuity sets, typically it does not allow corners or multiple points (i.e. self intersections of the discontinuities set). Relaxation of this assumption was already discussed in this earlier work. We highlighted that it would imply to consider other estimators, as, even in the noiseless setting, the histogram approximation falls short on simple examples (see Figure \ref{fig:reach }). This serves as a motivation to depart from plug-in approaches and as the starting point for this new work.\\\\
\noindent We consider the Gaussian white noise model given by the following stochastic equation,
\begin{equation*}
\label{white noise model}
    dX_{t_{1},...,t_{d}}=f(t_{1},...,t_{d})dt_{1}...dt_{d}+\theta dW_{t_{1},...,t_{d}}
\end{equation*}
with $W$ a $d-$parameters Wiener field, $f:[0,1]^{d}\rightarrow \mathbb{R}$ a signal and $\theta\geq0$ the level of noise. In this context, our goal is to estimate $\operatorname{dgm}(f)$, the persistence diagram coming from the sublevel sets of $f$.\\\\
For a set $A\subset[0,1]^{d}$, $\overline{A}$ we denote its adherence and $\partial A$ its boundary. We suppose that $f$ verifies the following assumptions :
\begin{itemize}
\item \textbf{A1.} f is a piecewise constant function, i.e. there exist $M_{1},...,M_{l}$ open sets of $[0,1]^{d}$ and $\lambda_{1}$, ..., $\lambda_{l}$ in $\mathbb{R}$ such that $\bigcup_{i=1}^{l}\overline{M_{i}}=[0,1]^{d}$ and $f|_{M_{i}}=\lambda_{i}$, $\forall i\in\{1,...,l\}$.
\item \textbf{A2}. $f$ verifies, $\forall x_{0}\in [0,1]^{d}$,
$$\underset{x\in \bigcup\limits_{i=1}^{l}M_{i}\rightarrow x_{0}}{\liminf}f(x)=f(x_{0}).$$
In this context, two signals, differing only on a null set, are statistically undistinguishable. Persistent homology is sensitive to point-wise irregularity, two signals differing only on a null set can have persistence diagrams that are arbitrarily far. Assumption \textbf{A2} prevents such scenario. 
\item \textbf{A3.} Let $\mu\in]0,1]$ and $R_{\mu}>0$, for all $I\subset \{1,...,l\}$
$$\operatorname{reach}_{\mu}\left(\bigcup\limits_{i\in I}\partial M_{i}\right)\geq R_{\mu}$$
with $\operatorname{reach}_{\mu}$ denoting the $\mu-$reach.
\end{itemize}
The class of functions verifying \textbf{A1}, \textbf{A2} and \textbf{A3} is denoted $S_{d}(\mu,R_{\mu})$. For our purpose, we show in Appendix \ref{appendix tameness} that persistence diagrams of signals in $S_{d}(\mu,R_{\mu})$ are well-defined. Compared to the classes, we introduced \cite{Henneuse24a}, our analytic assumption over $f$ on the regular region are stronger, but we significantly relax the assumption on the geometry of the discontinuities allowing multiple points and corners in the discontinuities set.
\subsection*{Contribution}
We complete the main result of \cite{Henneuse24a}. Unlike the classes considered there, over the classes $S_{d}(\mu,R_{\mu})$, a plug-in estimator from histogram estimation is no longer consistent (see Figure \ref{fig:reach }). To overcome this issue, we propose, a non-plug-in estimator of the persistence diagram of $f$, based on image persistence \citep{CSEHM2009}, denoted $\widehat{\operatorname{dgm}(f)}$. Over $S_{d}(\mu,R_{\mu})$, we show that this estimator is consistent and achieves parametric convergence rates. More precisely, our main result is the following theorem.
\begin{customthm}{1}
\label{th1}
There exist $\Tilde{C_{0}}$ and $\Tilde{C_{1}}$ such that, for all $t>0$,
$$\mathbb{P}\left(\sup\limits_{f\in S_{d}(\mu,R_{\mu})}d_{b}\left(\widehat{\operatorname{dgm}(f)},\operatorname{dgm}(f)\right)\geq t\theta\right)\leq \Tilde{C_{0}}\exp\left(-\Tilde{C_{1}}t^{2}\right).$$ 
\end{customthm}
\noindent The paper is organized as follows. Section \ref{Background section} recalls the necessary background on geometric measure theory and persistent homology, Section \ref{Procedure description sec} describes our estimation procedure, and Section \ref{proof TH1} is dedicated to the proof of Theorem \ref{th1}. Proofs of technical lemmas can be found in appendix.
\section{Background}
\label{Background section}
This section provides the necessary background to follow this paper.
\subsection{Persistent Homology}
\label{Background section 2}
We here present briefly some notions related to persistence homology. For a broader overview and visual illustrations of persistent homology, we recommend \cite{chazal2021introduction}. For detailed and rigorous constructions, see \cite{chazal2013}. Additionally, since the construction discussed here involves (singular) homology, the reader can refer to \cite{Hatcher}.
\begin{defi}
Let $\Lambda \subset \mathbb{R}$ be a set of indices. A \textbf{filtration} over $\Lambda$ is a family $\left(\mathcal{K}_\lambda\right)_{\lambda \in \Lambda}$ of topological spaces satisfying, $\forall \lambda, \lambda^{\prime} \in \Lambda, \lambda \leqslant \lambda^{\prime}$,
$$
\mathcal{K}_\lambda \subset \mathcal{K}_{\lambda^{\prime}}
.$$
\end{defi}
\noindent A typical filtration that we will consider in this paper is, for a function $f:\mathbb{R}^{d}\rightarrow\mathbb{R}$, the family of sublevel sets $\left(\mathcal{F}_{\lambda}\right)_{\lambda\in\mathbb{R}}=(f^{-1}(]-\infty,\lambda]))_{\lambda\in\mathbb{R}}$. The associated family of homology groups of degree $s\in \mathbb{N}$, $\mathbb{V}_{f,s}=\left(H_{s}\left(\mathcal{F}_{\lambda}\right)\right)_{\lambda\in\mathbb{R}}$, equipped with $v_{\lambda}^{\lambda^{'}}$ the linear application induced by the inclusion $\mathcal{F}_{\lambda}\subset\mathcal{F}_{\lambda^{\prime}}$, for all $\lambda\leq \lambda^{'}$, forms a persistence module. To be more precise, in this paper, $H_{s}(.)$ is the singular homology functor in degree $s$ with coefficients in a field (typically $\mathbb{Z}/2\mathbb{Z}$). Hence, $H_{s}\left(\mathcal{F}_{\lambda}\right)$ is a vector space. \\\\
Still, the estimator we propose in Section \ref{Procedure description sec}, relies on image module, that does not come (directly) from sublevel sets filtration. Hence, we need the following, more general, definition.
\begin{defi}
 Let $\Lambda\subset \mathbb{R}$ be a set of indices. A \textbf{persistence module} over $\lambda$ is a family $\mathbb{V}=\left(\mathbb{V}_\lambda\right)_{\lambda \in \Lambda}$ of vector spaces equipped with linear application $v_\lambda^{\lambda^{\prime}}: \mathbb{V}_\lambda \rightarrow \mathbb{V}_{\lambda^{\prime}}$ such that, $\forall \lambda \leqslant \lambda^{\prime} \leqslant \lambda^{\prime \prime} \in \Lambda$,
$$ v_\lambda^\lambda=i d$$
and 
$$v_{\lambda^{\prime}}^{\lambda^{\prime \prime}} \circ v_\lambda^{\lambda^{\prime}}=v_\lambda^{\lambda^{\prime \prime}}.$$
\end{defi}
\noindent Under $q-$tameness of the persistence module, it is possible to show that the algebraic structure of the persistence module encodes exactly the evolution of the topological features along the indices $\Lambda$.
\begin{defi}
\label{def: q-tame}
A persistence module $\mathbb{V}$ is said to be \textbf{$q$-tame} if $\forall \lambda<\lambda^{\prime} \in \Lambda, \operatorname{rank}\left(v_\lambda^{\lambda^{\prime}}\right)$ is finite. By extension, when considering the persistence modules $(\mathbb{V}_{f,s})_{s\in\mathbb{N}}$ coming from the sublevels sets filtration of a real function $f$, we say that $f$ is $q-$tame if $\mathbb{V}_{f,s}$ is for all $s\in\mathbb{N}$.
\end{defi}
\noindent Furthermore, the algebraic structure of the persistence module can be summarized by a collection $\{(b_{i},d_{i}), i\in I\}\subset\overline{\mathbb{R}}^{2}$, which defines the \textbf{persistence diagram}. Following previous remarks, $b_{i}$ corresponds to the birth time of a topological feature, $d_{i}$ to its death time and $d_{i}-b_{i}$ to its lifetime. For a detailed construction of persistence diagrams, see \cite{chazal2013}. 

\begin{figure}[H]
\centering
\begin{subfigure}{.5\textwidth}
  \centering
  \includegraphics[scale=0.5]{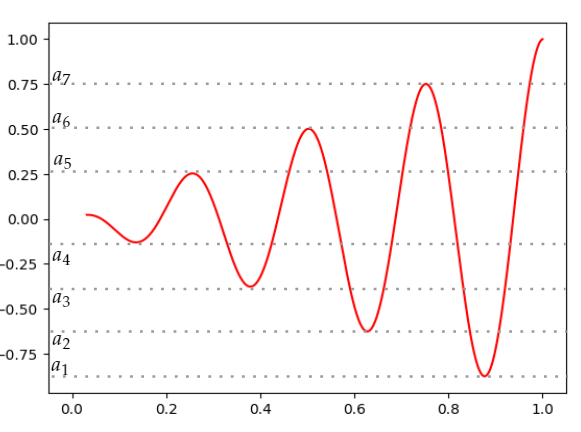}
  \caption{graph of $f$}
  \label{fig:sub1}
\end{subfigure}%
\begin{subfigure}{.5\textwidth}
  \centering
  \includegraphics[scale=0.5]{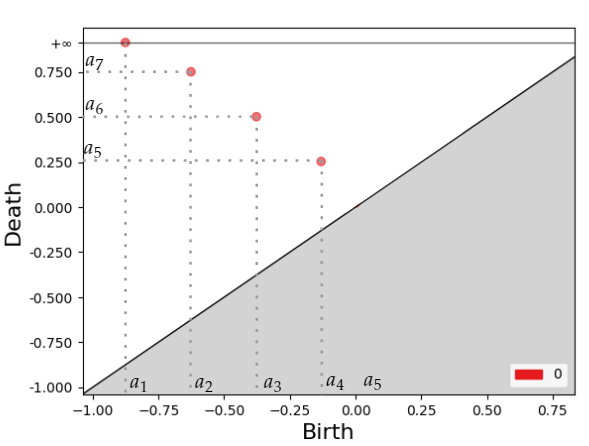}
  \caption{$H_{0}$-persistence diagram of $f$}
  \label{fig:sub2}
\end{subfigure}
\caption{Graph of $f(x)=x\cos(8\pi x)$ over $[0,1]$ and the persistence diagram associated to its sublevel sets filtration. $a_{1}$, ..., $a_{4}$ correspond to local minima of $f$ and thus birth times in $\operatorname{dgm}(f)$. $a_{5}$, ..., $a_{7}$ correspond to local maxima of $f$ and thus death times in $\operatorname{dgm}(f)$.}
\label{fig:test}
\end{figure}

\begin{figure}[H]
\centering
\begin{subfigure}{.5\textwidth}
  \centering
  \includegraphics[scale=0.3]{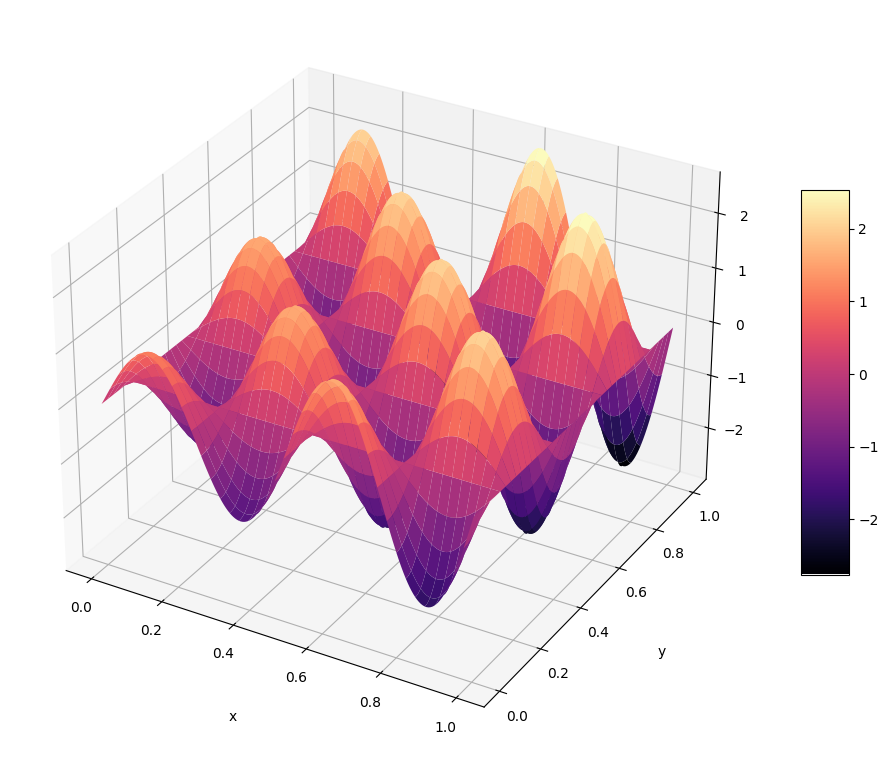}
  \caption{graph of $f$}
  \label{fig:sub1}
\end{subfigure}%
\begin{subfigure}{.5\textwidth}
  \centering
  \includegraphics[scale=0.5]{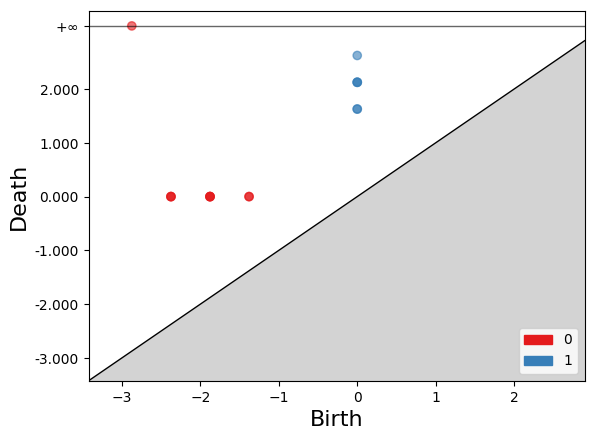}
  \caption{persistence diagram of $f$}
  \label{fig:sub2}
\end{subfigure}
\caption{Graph of $f(x)=\sin(4\pi x)\cos(4\pi y)(x+y+1)$ over $[0,1]^{2}$ and the persistence diagram associated to its sublevel sets filtration, red points corresponds to the $H_{0}-$persistence diagram and blue points to the $H_{1}-$persistence diagram.}
\label{fig:test}
\end{figure}

\noindent To compare persistence diagrams, a popular distance, especially in statistical contexts, is the bottleneck distance.
\begin{defi}
\label{def: bottleneck}
The \textbf{bottleneck distance} between two persistence diagrams $D_{1}$ and $D_{2}$ is,
$$d_{b}\left(D_{1},D_{2}\right)=\underset{\gamma\in\Gamma}{\inf}\underset{p\in D_{1}}{\sup}||p-\gamma(p)||_{\infty}$$
    with $\Gamma$ the set of all bijections between $D_{1}$ and $D_{2}$ (both enriched with the diagonal).
\end{defi}
\noindent We now introduce the algebraic stability theorem for the bottleneck distance. This theorem was the key for proving upper bounds in \cite{Henneuse24a}, it also plays a crucial role, here, in proving Theorem \ref{MainProp}. This theorem relies on interleaving between modules, a notion we extensively exploit.
\begin{defi}
\label{def:interleaving}
Two persistence modules $\mathbb{V}=\left(\mathbb{V}_{\lambda}\right)_{\lambda\in \Lambda\subset \mathbb{R}}$ and $\mathbb{W}=\left(\mathbb{W}_{\lambda}\right)_{\lambda\in \Lambda\subset \mathbb{R}}$ are said to be \textbf{$\varepsilon$-interleaved} if there exist two families of applications $\phi=\left(\phi_{\lambda}\right)_{\lambda\in \Lambda\subset \mathbb{R}}$ and $\psi=\left(\psi_{\lambda}\right)_{\lambda\in \Lambda\subset \mathbb{R}}$ where $\phi_{\lambda}:\mathbb{V}_{\lambda}\rightarrow\mathbb{W}_{\lambda+\varepsilon}$, $\psi_{\lambda}:\mathbb{W}_{\lambda}\rightarrow\mathbb{V}_{\lambda+\varepsilon}$, and for all $\lambda<\lambda^{'}$ the following diagrams commute,
\begin{center}
\begin{tikzcd}
	{\mathbb{V}_{\lambda }} && {\mathbb{V}_{\lambda ^{'}}} & {\mathbb{W}_{\lambda }} && {\mathbb{W}_{\lambda ^{'}}} \\
	{\mathbb{W}_{\lambda +\varepsilon }} && {\mathbb{W}_{\lambda '+\varepsilon }} & {\mathbb{V}_{\lambda +\varepsilon }} && {\mathbb{V}_{\lambda '+\varepsilon }} \\
	{\mathbb{V}_{\lambda }} && {\mathbb{V}_{\lambda+2\varepsilon }} & {\mathbb{W}_{\lambda}} && {\mathbb{W}_{\lambda+2\varepsilon }} \\
	& {\mathbb{W}_{\lambda +\varepsilon }} &&& {\mathbb{V}_{\lambda +\varepsilon }}
	\arrow["{\phi _{\lambda }}"', from=1-1, to=2-1]
	\arrow["{\phi _{\lambda ^{'}}}", from=1-3, to=2-3]
	\arrow["{w_{\lambda +\varepsilon }^{\lambda ^{'} +\varepsilon }}"', from=2-1, to=2-3]
	\arrow["{w_{\lambda }^{\lambda ^{'}}}", from=1-4, to=1-6]
	\arrow["{\psi_{\lambda }}"', from=1-4, to=2-4]
	\arrow["{v_{\lambda +\varepsilon }^{\lambda ^{'} +\varepsilon }}"', from=2-4, to=2-6]
	\arrow["{\psi_{\lambda ^{'}}}", from=1-6, to=2-6]
	\arrow["{v_{\lambda }^{\lambda ^{'}+2\varepsilon}}", from=3-1, to=3-3]
	\arrow["{\phi _{\lambda }}"', from=3-1, to=4-2]
	\arrow["{\psi_{\lambda+\varepsilon}}"', from=4-2, to=3-3]
	\arrow["{\psi_{\lambda }}"', from=3-4, to=4-5]
	\arrow["{v_{\lambda }^{\lambda ^{'}}}", from=1-1, to=1-3]
	\arrow["{w_{\lambda }^{\lambda ^{'}+2\varepsilon}}", from=3-4, to=3-6]
	\arrow["{\phi _{\lambda+\varepsilon}}"', from=4-5, to=3-6]
\end{tikzcd}
\end{center}
\end{defi}
\begin{customthm}{ (\citealt[]["algebraic stability"]{Chazal2009})}
Let $\mathbb{V}$ and $\mathbb{W}$ two $q-$tame persistence modules. If  $\mathbb{V}$ and $\mathbb{W}$  are $\varepsilon-$interleaved then,
$$d_{b}\left(\operatorname{dgm}(\mathbb{V}),\operatorname{dgm}(\mathbb{W})\right)\leq \varepsilon.$$
\end{customthm}
\noindent We now give a corollary of this result, proved earlier in special cases \citep{Baranikov94,CSEH2005}. We insist on the fact that this is a strictly weaker result than algebraic stability.
\begin{customthm}{("sup norm stability")}
    Let $f$ and $g$ two real-valued $q$-tame function, for all $s\in\mathbb{N}$
$$d_{b}\left(\operatorname{dgm}\left(\mathbb{V}_{f,s}\right),\operatorname{dgm}\left(\mathbb{V}_{g,s}\right)\right)\leq ||f-g||_{\infty}.$$
\end{customthm}
\begin{figure}[h]
\centering
\begin{subfigure}{0.5\textwidth}
  \centering
  \includegraphics[scale=0.45]{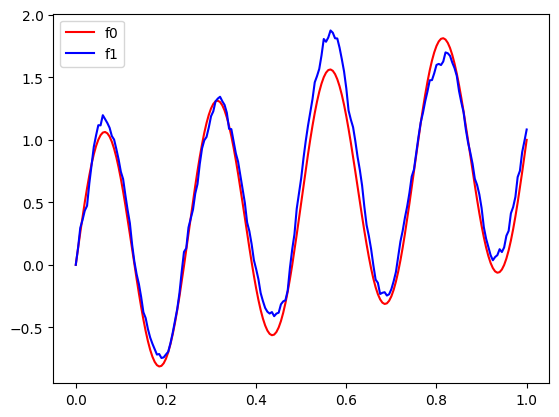}
  \caption{graphs of $f_0$ and $f_1$}
  \label{fig:sub1}
\end{subfigure}%
\begin{subfigure}{0.5\textwidth}
  \centering
  \includegraphics[scale=0.5]{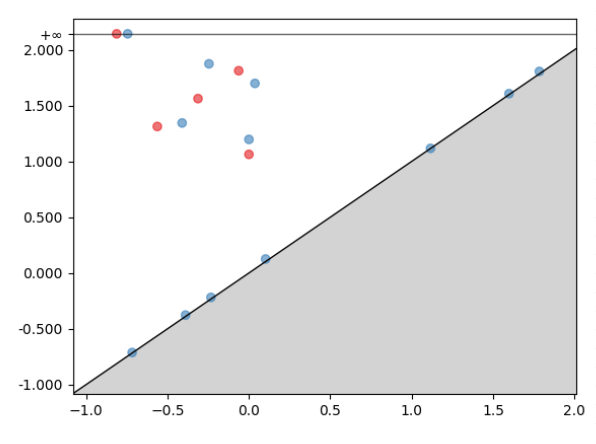}
  \caption{persistence diagrams of $f_0$ and $f_1$}
  \label{fig:sub2}
\end{subfigure}
\caption{1D Illustration of stability theorems.}
\label{fig:test}
\end{figure}
\noindent This last property is often used to upper bound the errors (in bottleneck distance) of "plug-in" estimators of persistence diagrams. It enables the direct translation of convergence rates in sup-norm  to convergence rates in bottleneck distance, which for regular classes of signals (e.g. Hölder spaces) provides minimax upper bounds. Still, for wider classes, this approach falls short. The alternative approach proposed in this work focuses instead on the estimation of the persistent modules $\mathbb{V}_{f,s}$, in the sense of interleaving, and then exploits algebraic stability.
\subsection{Generalized gradient}
\label{Background section 1}
We here present some concepts from geometric measure theory, extensively used in geometric inference and TDA. The first notion used in this paper is the distance function to a compact set.
\begin{defi}
Let $K\subset[0,1]^{d}$ a compact set, the \textbf{distance function} $d_{K}$ is given by,
$$d_{K}:x\mapsto \min\limits_{y\in K}||x-y||_{2}.$$
\end{defi}
\noindent Generally, the distance function is not differentiable everywhere, still it is possible to define a generalized gradient function that matches the gradient at points where the distance function is differentiable.
\begin{defi}
Let,
$$\Gamma_K(x)=\{y \in K \mid ||x-y||_{2}=d_{K}(x)\}$$
the set of closest points to $x$ in $K$. For $x\in[0,1]^{d}\setminus K$, let denote $\Theta_{K}(x)$ the center of the unique smallest ball enclosing $\Gamma_K(x)$, the \textbf{generalized gradient function} $\nabla_{K}(x)$ is defined as,
$$\nabla_{K}(x)=\frac{x-\Theta_{K}(x)}{d_{k}(x)}.$$
\end{defi}
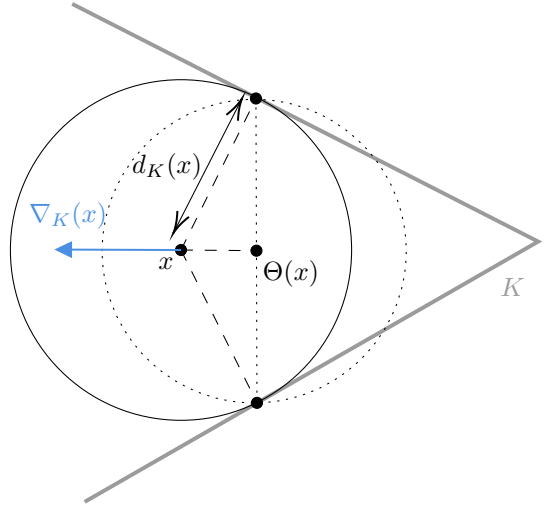
\begin{figure}[h]
    \centering
    \begin{tikzpicture}[x=0.75pt,y=0.75pt,yscale=-1,xscale=1]

%Straight Lines [id:da8537445711353691] 
\draw [color={rgb, 255:red, 155; green, 155; blue, 155 }  ,draw opacity=1 ][line width=1.5]    (137.98,265.87) -- (366.67,134.5) -- (131.67,14.53) ;
%Shape: Ellipse [id:dp01416969206433949] 
\draw   (100.5,138.76) .. controls (100.5,91.28) and (139.02,52.8) .. (186.54,52.8) .. controls (234.06,52.8) and (272.59,91.28) .. (272.59,138.76) .. controls (272.59,186.23) and (234.06,224.71) .. (186.54,224.71) .. controls (139.02,224.71) and (100.5,186.23) .. (100.5,138.76) -- cycle ;
%Shape: Ellipse [id:dp35373233592737874] 
\draw  [fill={rgb, 255:red, 0; green, 0; blue, 0 }  ,fill opacity=1 ] (183.76,138.76) .. controls (183.76,137.22) and (185,135.97) .. (186.54,135.97) .. controls (188.08,135.97) and (189.33,137.22) .. (189.33,138.76) .. controls (189.33,140.29) and (188.08,141.54) .. (186.54,141.54) .. controls (185,141.54) and (183.76,140.29) .. (183.76,138.76) -- cycle ;
%Shape: Ellipse [id:dp3228471392103325] 
\draw  [fill={rgb, 255:red, 0; green, 0; blue, 0 }  ,fill opacity=1 ] (221.55,62.29) .. controls (221.55,60.75) and (222.8,59.5) .. (224.34,59.5) .. controls (225.88,59.5) and (227.13,60.75) .. (227.13,62.29) .. controls (227.13,63.82) and (225.88,65.07) .. (224.34,65.07) .. controls (222.8,65.07) and (221.55,63.82) .. (221.55,62.29) -- cycle ;
%Shape: Ellipse [id:dp6647047110023592] 
\draw  [fill={rgb, 255:red, 0; green, 0; blue, 0 }  ,fill opacity=1 ] (222.07,215.9) .. controls (222.07,214.36) and (223.32,213.12) .. (224.86,213.12) .. controls (226.4,213.12) and (227.64,214.36) .. (227.64,215.9) .. controls (227.64,217.44) and (226.4,218.68) .. (224.86,218.68) .. controls (223.32,218.68) and (222.07,217.44) .. (222.07,215.9) -- cycle ;
%Straight Lines [id:da5258463166844678] 
\draw  [dash pattern={on 4.5pt off 4.5pt}]  (186.54,138.76) -- (224.86,215.9) ;
%Straight Lines [id:da0563877878623531] 
\draw  [dash pattern={on 4.5pt off 4.5pt}]  (186.54,138.76) -- (224.34,62.29) ;
%Shape: Ellipse [id:dp22357123260140233] 
\draw  [dash pattern={on 0.84pt off 2.51pt}] (146.79,139.31) .. controls (146.79,97.01) and (181.12,62.72) .. (223.46,62.72) .. controls (265.81,62.72) and (300.13,97.01) .. (300.13,139.31) .. controls (300.13,181.61) and (265.81,215.9) .. (223.46,215.9) .. controls (181.12,215.9) and (146.79,181.61) .. (146.79,139.31) -- cycle ;
%Straight Lines [id:da19025543016662305] 
\draw  [dash pattern={on 0.84pt off 2.51pt}]  (224.34,62.29) -- (224.86,215.9) ;
%Shape: Ellipse [id:dp23560237641714377] 
\draw  [fill={rgb, 255:red, 0; green, 0; blue, 0 }  ,fill opacity=1 ] (221.81,139.09) .. controls (221.81,137.56) and (223.06,136.31) .. (224.6,136.31) .. controls (226.14,136.31) and (227.39,137.56) .. (227.39,139.09) .. controls (227.39,140.63) and (226.14,141.88) .. (224.6,141.88) .. controls (223.06,141.88) and (221.81,140.63) .. (221.81,139.09) -- cycle ;
%Straight Lines [id:da2383532852607555] 
\draw [color={rgb, 255:red, 74; green, 144; blue, 226 }  ,draw opacity=1 ][line width=0.75]    (186.54,138.76) -- (125.67,138.54) ;
\draw [shift={(122.67,138.53)}, rotate = 0.2] [fill={rgb, 255:red, 74; green, 144; blue, 226 }  ,fill opacity=1 ][line width=0.08]  [draw opacity=0] (8.93,-4.29) -- (0,0) -- (8.93,4.29) -- cycle    ;
%Straight Lines [id:da6507541014515408] 
\draw  [dash pattern={on 4.5pt off 4.5pt}]  (186.54,138.76) -- (224.6,139.09) ;
%Straight Lines [id:da4608935250974129] 
\draw    (216.42,64.05) -- (183.58,128.48) ;
\draw [shift={(182.67,130.27)}, rotate = 297.01] [color={rgb, 255:red, 0; green, 0; blue, 0 }  ][line width=0.75]    (10.93,-3.29) .. controls (6.95,-1.4) and (3.31,-0.3) .. (0,0) .. controls (3.31,0.3) and (6.95,1.4) .. (10.93,3.29)   ;
\draw [shift={(217.33,62.27)}, rotate = 117.01] [color={rgb, 255:red, 0; green, 0; blue, 0 }  ][line width=0.75]    (10.93,-3.29) .. controls (6.95,-1.4) and (3.31,-0.3) .. (0,0) .. controls (3.31,0.3) and (6.95,1.4) .. (10.93,3.29)   ;

% Text Node
\draw (346.41,151.58) node [anchor=north west][inner sep=0.75pt]  [font=\small,color={rgb, 255:red, 155; green, 155; blue, 155 }  ,opacity=1 ]  {$K$};
% Text Node
\draw (173.88,141.49) node [anchor=north west][inner sep=0.75pt]  [font=\small]  {$x$};
% Text Node
\draw (226.6,142.49) node [anchor=north west][inner sep=0.75pt]  [font=\small]  {$\Theta ( x)$};
% Text Node
\draw (108.67,113.47) node [anchor=north west][inner sep=0.75pt]  [font=\small,color={rgb, 255:red, 74; green, 144; blue, 226 }  ,opacity=1 ]  {$\nabla _{K}( x)$};
% Text Node
\draw (160.67,88.4) node [anchor=north west][inner sep=0.75pt]  [font=\small]  {$d_{K}( x)$};
\end{tikzpicture}
    \caption{2D example with 2 closest points}
\end{figure}
\noindent With these definitions we can now introduce the notion of $\mu-$reach \citep{mureach}.
\begin{defi}
Let $K \subset [0,1]^{d}$ a compact set, its \textbf{$\mu$-reach} $\operatorname{reach}_\mu(K)$ is defined by,
$$
\operatorname{reach}_\mu(K)=\inf \left\{r \mid \inf _{d_K^{-1}(r)\setminus K}\left\|\nabla_K\right\|_{2}<\mu\right\}.
$$
\end{defi}
\noindent The $1-$reach corresponds simply to the reach, the curvature measure introduced by \cite{Fed59}, involved in our earlier work \cite{Henneuse24a}. The $\mu-$reach is positive for a large class of sets, indeed the union for all $\mu>0$ of compact sets with positive $\mu$-reach is dense in the space of compact sets (for the Hausdorff distance). In particular, for our purpose, Assumption \textbf{A3} allows considering discontinuity sets displaying corners and multiple points (i.e. self-intersections), cases that were excluded in \cite{Henneuse24a}.\\\\
We now state a particular property of sets with positive $\mu-$reach from \cite{SSDO07}.  This property is used in the proof of Theorem \ref{MainProp} to quantify approximation errors. For a set $K\subset[0,1]^{d}$ and $r\geq 0$ we denote $B_{2}(K,r)=\{x\in[0,1]^{d}\text{ s.t. }d_{K}(x)<r\}$ and $\overline{B}_{2}(K,r)=\{x\in[0,1]^{d}\text{ s.t. }d_{K}(x)\leq r\}$.
\begin{lmm}{\cite[Lemma 3.1.]{SSDO07}}
\label{lmm SSDO7}
Let $K\subset [0,1]^{d}$ a compact set and let $\mu>0, r>0$ be such that $r< \operatorname{reach}_\mu(K)$. For any $x \in \overline{B}_{2}(K,r) \backslash K$, one has
$$
d_{ \partial \overline{B}_{2}(K,r)}\left(x\right) \leq \frac{r-d_K(x)}{\mu} \leq \frac{r}{\mu}.
$$
\end{lmm}
\section{Procedure description}
\label{Procedure description sec}
We here propose a two-steps estimation procedure.\\
The first step is a "rough" estimation of sublevel sets via local averaging on a regular grid plus a thickening. We call these estimators rough, as they are not meant to perform well for the estimation of sublevel sets evaluated with standard metrics. Typically, over $S_{d}(\mu,R_{\mu})$, for a fixed $\lambda\in\mathbb{R}$, our estimator $\widehat{\mathcal{F}}_{\lambda}$ can be (arbitrarily) far in Hausdorff distance from $\mathcal{F}_{\lambda}$. But, it captures well the persistent topological features of the true signal.\\
Still, cubical approximation and noise may create cycles lying around the boundaries of sublevel sets estimators that do not correspond to any cycles of the true signal (as in Figure \ref{fig:reach }.b). These cycles are pretty benign when appearing in regular regions (they will have "short" lifetime), but may have arbitrarily long lifetime when appearing around the discontinuities of the signal. Hence, for the second step, in the same spirit that \cite{Bobrowski}, instead of considering directly the modules coming from this estimated filtration, we construct an image persistence module. Again, this can be seen as a topological regularization step : it aims to eliminate such cycles (or at least make their lifetime short) without damaging (too much) the information inferred about true cycles.\\
Note that the use of image modules to filter out topological noise is not a new idea, early works by \cite{CSEHM2009} and \cite{Chazal11} already use image module to estimate persistence diagrams of signals while working on noisy domains.\\\\
In the following, for a set $A\subset \mathbb{R}^{d}$ and $b\geq 0$, we denote,
$$A^{b}=\left\{x\in \mathbb{R}^{d} \text{ s.t. } d_{\infty}\left(x,A\right)\leq b\right\}$$ 
with 
$$d_{\infty}\left(x,A\right)=\inf\limits_{y\in A}||x-y||_{\infty}.$$
\begin{figure}[H]
    \centering
    \subfloat[\centering A true cycle not captured by cubical approximation]{{\includegraphics[width=5.5cm]{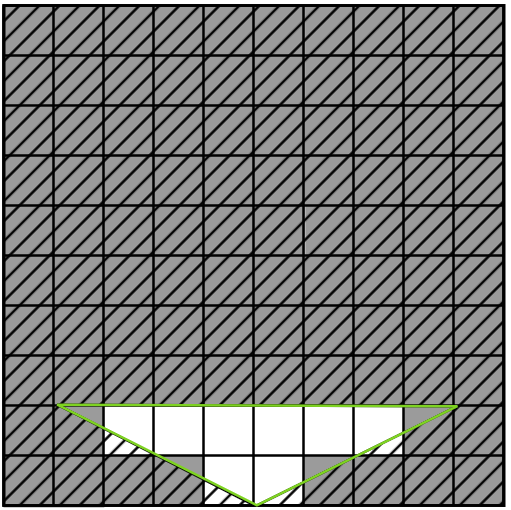} }}%
    \qquad
    \subfloat[\centering A false cycle created by cubical approximation]{{\includegraphics[width=5.5cm]{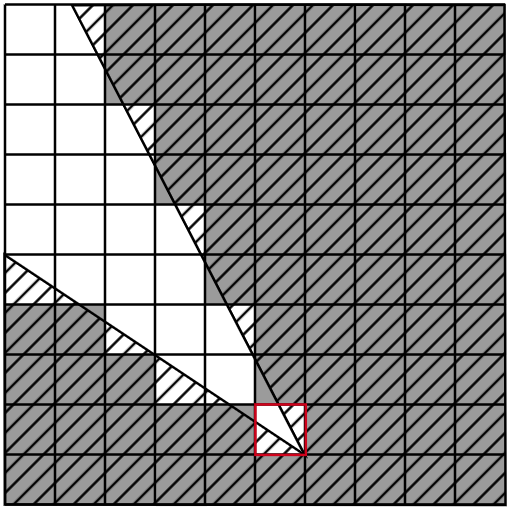} }}%
    \caption{$\lambda-$sublevel cubical approximation for $f$ the function defined as $0$ on the hatched area and $K$ outside (for arbitrarily large $K$). (a) display a case where the histogram approximation fails to capture true cycle in green, for, at least, all $0<\lambda<K/2$. (b) displays a case where the histogram approximation creates a cycle in red, not corresponding to any true cycle, with an arbitrarily long lifetime.  }%
    \label{fig:reach }%
\end{figure}
\noindent \textbf{Step 1 : rough sublevel sets estimation.} Let $h>0$ such that $1/h$ is an integer, consider $G_{h}$ the regular orthogonal grid over $[0,1]^{d}$ of step $h$ and $C_{h}$ the collection of all the hypercubes of side $h$ composing $G_{h}$. Let $r_{1}>0$, we define, for all $\lambda\in\mathbb{R}$, the "rough" $\lambda-$sublevel estimator,
$$\widehat{\mathcal{F}}_{\lambda}=\left(\bigcup\limits_{H\in C_{h,\lambda}}H\right)^{\lceil r_{1}\rceil h} \text{ with }C_{h,\lambda}=\left\{H\in C_{h}\text{ such that }\int_{H}dX-\int_{H}\lambda\leq 0\right\}$$
and $\lceil.\rceil$ the ceiling function. This thickening of the sublevel sets obtained by histogram estimation of $f$ aims to solve the problem exposed in Figure \ref{fig:reach }.a. \\\\
\textbf{Step 2 : construction of the image persistence module and associated diagram.} Let, $r_{2}>0$ and,
$$\rho_{\lambda}:H_{s}\left(\widehat{\mathcal{F}}_{\lambda}\right)\rightarrow H_{s}\left(\widehat{\mathcal{F}}_{\lambda}^{\lceil r_{2}\rceil h}\right)$$
the map induced by the inclusion $\widehat{\mathcal{F}}_{\lambda}\subset \widehat{\mathcal{F}}_{\lambda}^{\lceil r_{2}\rceil h}$. Then denote, $$\hat{v}_{\lambda}^{\lambda^{'}}:H_{s}\left(\widehat{\mathcal{F}}_{\lambda}^{\lceil r_{2}\rceil h}\right)\rightarrow H_{s}\left(\widehat{\mathcal{F}}_{\lambda^{'}}^{\lceil r_{2}\rceil h}\right)$$
the map induced by the inclusion.
$\widehat{\mathcal{F}}_{\lambda}^{\lceil r_{2}\rceil h}\subset \widehat{\mathcal{F}}_{\lambda^{'}}^{\lceil r_{2}\rceil h}$. Note that the following diagram commutes (all maps are induced by inclusions),
\begin{equation*}
\label{diagBase}
\begin{tikzcd}
	{H_{s}\left(\widehat{\mathcal{F}}_{\lambda}\right)} &&& {H_{s}\left(\widehat{\mathcal{F}}_{\lambda^{'}}\right)} \\
	\\
	{H_{s}\left(\widehat{\mathcal{F}}_{\lambda}^{\lceil r_{2}\rceil h}\right)} &&& {H_{s}\left(\widehat{\mathcal{F}}_{\lambda^{'}}^{\lceil r_{2}\rceil h}\right)}
	\arrow[from=1-1, to=1-4,]
	\arrow[from=1-1, to=3-1,"\rho_{\lambda}"]
	\arrow[from=1-4, to=3-4,"\rho_{\lambda^{'}}"]
	\arrow[from=3-1, to=3-4,"\hat{v}_{\lambda}^{\lambda^{'}}"]
\end{tikzcd}
\end{equation*}
Thus for all $\lambda<\lambda^{'}$, 
$$\widehat{v}_{\lambda,h}^{\lambda^{'}}\left(\operatorname{Im}(\rho_{\lambda})\right)\subset\operatorname{Im}(\rho_{\lambda^{'}}).$$
We now introduce $\widehat{\mathbb{V}}_{f,s}$ the persistence module associated to $(\operatorname{Im}(\rho_{\lambda}))_{\lambda\in \mathbb{R}}$ equipped with the collection of maps $(\widehat{v}_{\lambda,h}^{\lambda^{'}})_{\lambda<\lambda^{'}}$ for the $s$-th order homology, $\widehat{\operatorname{dgm}_{s}(f)}$ the associated persistence diagram, and $\widehat{\operatorname{dgm}(f)}$ the collection of such persistence diagrams for all $s\in\mathbb{N}$. This diagram is well-defined, as we prove in Appendix \ref{appendix tameness} that $\widehat{\mathbb{V}}_{f,s}$ is $q-$tame. This second step aims to solve the problem exposed in Figure \ref{fig:reach }.b.\\\\
\textbf{Computation. }By construction, for all $\lambda\in\mathbb{R}$, $\widehat{\mathcal{F}}_{\lambda}$ and $\widehat{\mathcal{F}}_{\lambda}^{\lceil r_{2}\rceil h}$ are simply unions of cubes from the regular grid $G_{h}$, thus can be thought as (geometric realization of) cubical complexes (or even simplicial complexes). Hence, the computation of $\widehat{\operatorname{dgm}(f)}$ is  made possible by the algorithm for image persistence gave in \cite{CSEHM2009}.\\\\
Another strategy, would be to consider,
$$\widehat{\mathcal{F}}_{\lambda}=\overline{B}_{2}\left(\bigcup\limits_{H\in C_{h,\lambda}}c_{H},r_{1}h\right)$$
with $c_{H}$ the center of the hypercube $H$ and $\hat{\mathbb{V}}_{f,s}$, the image module induced by,
$$\rho_{\lambda}:H_{s}\left(\widehat{\mathcal{F}}_{\lambda}\right)\rightarrow H_{s}\left(\overline{B}_{2}\left(\widehat{\mathcal{F}}_{\lambda},r_{2}h\right)\right).$$
One can show, adapting the proofs of Section \ref{proof TH1}, that the subsequent estimator achieve the same convergence rates (up to slight change in constant). As $\widehat{\mathcal{F}}_{\lambda}$ and  $B_{2}\left(\widehat{\mathcal{F}}_{\lambda},r_{2}h\right)$ are unions of Euclidean balls, they can be replaced by Čech complexes. To justify properly that it lead to the same estimated persistence diagram, one can adapt the "parameter nerve theorem" (Lemma 3.4) from \cite{ChazalOudot08}. Again, the computation of this persistence diagram is then made possible by the algorithm provided in \cite{CSEHM2009}.
\section{Convergence results}
\label{proof TH1}
This section is dedicated to the proof of Theorem \ref{MainProp}.
\subsection{Main proof}
The idea of proof follows similarly to the proof for upper bounds in \cite{Henneuse24a}. The general idea is to construct an interleaving between the module $\mathbb{V}_{f,s}$ and $\widehat{\mathbb{V}}_{f,s}$ and apply the algebraic stability theorem. This implies to construct two morphisms, $\overline{\psi}:\mathbb{V}_{f,s}\rightarrow\widehat{\mathbb{V}}_{f,s}$ and $\overline{\phi}:\widehat{\mathbb{V}}_{f,s}\rightarrow\mathbb{V}_{f,s}$ satisfying definition \ref{def:interleaving}. Propositions \ref{inclusion} and \ref{lemmaFiltEquiv} give the necessary ingredients.\\\\ 
Let, $h>0$ and,
$$||W||_{h}=\sup_{H\in C_{h}} \frac{|W(H)|}{h^{d}}$$
with $W(H)=\int_H d W$.
\begin{prp}
\label{inclusion}
Let $f\in S_{d}(\mu,R_{\mu})$, $h<R_{\mu}\mu/\sqrt{d}$ and $r_{1}=\sqrt{d}/\mu$. We have, for all $\lambda\in\mathbb{R}$,
$$\mathcal{F}_{\lambda-\theta||W||_{h}}\subset \widehat{\mathcal{F}}_{\lambda}\subset \mathcal{F}_{\lambda+\theta||W||_{h}}^{(\sqrt{d}+\lceil\sqrt{d}/\mu\rceil)h}.$$
\end{prp}
\noindent The proof of Proposition \ref{inclusion} can be found in Section \ref{proof inclusion}. The lower inclusion given by this proposition induces directly a morphism from $\mathbb{V}_{f,s}$ to $\widehat{\mathbb{V}}_{f,s}$, this will be our $\overline{\psi}$. But the upper inclusion is not sufficient to induce a converse morphism, the construction of $\overline{\phi}$ requires additional work. To be able to use this upper inclusion, we need a map from the homology groups of $(\mathcal{F}_{\lambda}^{h})_{\lambda\in\mathbb{R}}$ (for all sufficiently small $h>0$) into the homology groups of $(\mathcal{F}_{\lambda})_{\lambda\in\mathbb{R}}$. Proposition \ref{lemmaFiltEquiv} provides such a map, exploiting the geometrical properties of the discontinuities set imposed by Assumption \textbf{A3} to construct a deformation retraction from $\overline{B}_{2}(\mathcal{F}_{\lambda},h)$ to $\mathcal{F}_{\lambda}$.
\begin{defi}
A subspace $A$ of $X$ is called a \textbf{deformation retract} of $X$ if there is a continuous $F: X \times [0,1] \rightarrow X$ (called a homotopy) such that for all $x \in X$ and $a \in A$,
\begin{itemize}
    \item $F(x, 0)=x$ 
    \item $F(x, 1) \in A$
    \item $F(a, 1)=a$.
\end{itemize}
The function $F$ is called a \textbf{deformation retraction} from $X$ to $A$.
\end{defi}
\noindent Homotopy, and thus homology, is invariant under deformation retract. Thus, a deformation retraction from $X$ to $A$ induces isomorphism between homology groups. More precisely, for all $t\in[0,1]$, $F(.,t)$ induces a morphism $F(.,t)^{\#}:C_{s}(X)\rightarrow C_{S}(X)$ between $s-$chains of $X$ defined by composing each singular $s$-simplex $\sigma: \Delta^s \rightarrow X$ with $F(.,t)$ to get a singular $s$-simplex $F^{\#}(\sigma,t)=F(.,t)\circ\sigma: \Delta^s \rightarrow X$, then extending $F^{\#}(.,t)$ linearly via $F^{\#}\left(\sum_i n_i \sigma_i,t\right)=\sum_i n_i F^{\#}\left(\sigma_i,t\right)=$ $\sum_i n_i F(.,t)\circ \sigma_i$. The morphism $F^{*}(.,t):H_{s}(X)\rightarrow H_{s}(X)$ defined by $[C]\mapsto [F^{\#}(C,t)]$ can be shown to be an isomorphism for all $t\in[0,1]$ \citep[see][pages 110-113]{Hatcher}. In particular, $F^{*}(.,1):H_{s}(X)\rightarrow H_{s}(A)$ is an isomorphism.
\begin{prp}
\label{lemmaFiltEquiv}
    For all $0<h<R_{\mu}$ and all $\lambda\in\mathbb{R}$ there exists a map $F_{\lambda,h}:\overline{B}_{2}(\mathcal{F}_{\lambda},h)\times [0,1]\rightarrow \overline{B}_{2}(\mathcal{F}_{\lambda},h)$ such that, $F_{\lambda,h}$ is a deformation retraction of $\overline{B}_{2}(\mathcal{F}_{\lambda},h)$ onto $\mathcal{F}_{\lambda}$ and for all $x\in \overline{B}_{2}(\mathcal{F}_{\lambda},h)$ and all $t\in[0,1]$,
    $$F_{\lambda,h}(x,t)\in \overline{B}_{2}\left(x,2d_{\mathcal{F}_{\lambda}}(x)/\mu^{2}\right).$$
    
\end{prp}
 \noindent We now have all the necessary ingredients to construct $\overline{\phi}$ and prove Theorem \ref{MainProp}.
\begin{thm}
\label{MainProp}
Let $r_{1}=\sqrt{d}/\mu$ and $r_{2}=\sqrt{d}(1+2/\mu^{2})(\sqrt{d}+\lceil r_{1}\rceil)$ and $h=R_{\mu}/(2\lceil r_{2}\rceil)$, there exist $\Tilde{C_{0}}$ and $\Tilde{C_{1}}$ such that, for all $t>0$,
$$\mathbb{P}\left(\sup\limits_{f\in S_{d}(\mu,R_{\mu})}d_{b}\left(\widehat{\operatorname{dgm}(f)},\operatorname{dgm}(f)\right)\geq \theta t\right)\leq \Tilde{C_{0}}\exp\left(-\Tilde{C_{1}}t^{2}\right).$$ 
\end{thm}
\begin{proof}
Define the morphism,
$$\overline{\psi}_{\lambda}:H_{s}\left(\mathcal{F}_{\lambda}\right)\rightarrow H_{s}\left(\widehat{\mathcal{F}}_{\lambda+\theta||W||_{h}}\right)$$ 
induced by the inclusion $\mathcal{F}_{\lambda}\subset \widehat{\mathcal{F}}_{\lambda+\theta||W||_{h}}$ given by Proposition \ref{inclusion}.\\\\
For the second morphism, denote $k_{2}=k_{1}+\lceil r_{2}\rceil$ and consider, $$j_{1,\lambda}:\operatorname{Im}\left(\rho_{\lambda}\right)\rightarrow H_{s}\left(\overline{B}_{2}\left(\mathcal{F}_{\lambda+\theta||W||_{h}}, k_{2}h\right)\right)$$
the map induced by the inclusion $\widehat{\mathcal{F}}_{\lambda}^{\lceil  r_{2}\rceil h}\subset \overline{B}_{2}\left(\mathcal{F}_{\lambda+\theta||W||_{h}}, k_{2}h\right)$ given by Proposition \ref{inclusion}, and,
$$F^{*}_{\lambda+\theta||W||_{h},k_{2}h}:H_{s}\left(\overline{B}_{2}\left(\mathcal{F}_{\lambda+\theta||W||_{h}}, k_{2}h\right)\right)\rightarrow H_{s}\left(\mathcal{F}_{\lambda+\theta||W||_{h}}\right)$$
induced by the deformation retract of Proposition \ref{lemmaFiltEquiv}. We then define,
$$
\left\{
    \begin{array}{ll}
    \overline{\phi}_{\lambda}:\operatorname{Im}\left(\rho_{\lambda}\right)\rightarrow H_{2}\left(\mathcal{F}_{\lambda+\theta||W||_{h}}\right)\\
    \overline{\phi}_{\lambda}=F^{*}_{\lambda+\theta||W||_{h},k_{2}h}\circ j_{1,\lambda}
    \end{array}
\right.
$$
We now show that $\overline{\psi}$ and $\overline{\phi}$ induce an interleaving between $\widehat{\mathbb{V}}_{f,s}$ and $\mathbb{V}_{s,f}$. More precisely, we show that the following diagrams commute, for all $\lambda<\lambda^{'}$,
\begin{equation}
\label{diagram2}
\begin{tikzcd}
	{\operatorname{Im}\left(\rho_{\lambda}\right)} && {\operatorname{Im}\left(\rho_{\lambda^{'}}\right)} \\
	\\
	{H_{s}\left(\mathcal{F}_{\lambda+\theta||W||_{h}}\right)} && {H_{s}\left(\mathcal{F}_{\lambda^{'}+\theta||W||_{h}}\right)}
	\arrow[from=1-1, to=3-1,"\overline{\phi}_{\lambda}"]
	\arrow[from=1-1, to=1-3,"\hat{v}_{\lambda}^{\lambda^{'}}"]
	\arrow[from=1-3, to=3-3,"\overline{\phi}_{\lambda^{'}}"]
	\arrow[from=3-1, to=3-3,"v_{\lambda+\theta||W||_{h}}^{\lambda^{'}+\theta||W||_{h}}"]
\end{tikzcd}
\end{equation}
\begin{equation}
\label{diagram1}
\begin{tikzcd}
	{H_{s}\left(\mathcal{F}_{\lambda}\right)} &&& {H_{s}\left(\mathcal{F}_{\lambda^{'}}\right)} \\
	\\
	{\operatorname{Im}\left(\rho_{\lambda+\theta||W||_{h}}\right)} &&& {\operatorname{Im}\left(\rho_{\lambda^{'}+\theta||W||_{h}}\right)}
	\arrow[from=1-1, to=3-1,"\overline{\psi}_{\lambda}"]
	\arrow[from=1-1, to=1-4,"v_{\lambda}^{\lambda^{'}}"]
	\arrow[from=1-4, to=3-4,"\overline{\psi}_{\lambda^{'}}"]
	\arrow[from=3-1, to=3-4,"\hat{v}_{\lambda+\theta||W||_{h}}^{\lambda^{'}+\theta||W||_{h}}"]
\end{tikzcd}
\end{equation}
\begin{equation}
\label{diagram3}
\begin{tikzcd}
	{\operatorname{Im}\left(\rho_{\lambda}\right)} && {\operatorname{Im}\left(\rho_{\lambda+2\theta||W||_{h}}\right)} \\
	\\
	& {H_{s}\left(\mathcal{F}_{\lambda+\theta||W||_{h}}\right)}
	\arrow[from=1-1, to=3-2,"\overline{\phi}_{\lambda}"]
	\arrow[from=3-2, to=1-3,"\overline{\psi}_{\lambda+\theta||W||_{h}}"]
	\arrow[from=1-1, to=1-3,"\hat{v}_{\lambda}^{\lambda+2\theta||W||_{h}}"]
\end{tikzcd}
\end{equation}
\begin{equation}
\label{diagram4}
\begin{tikzcd}
	{H_{s}\left(\mathcal{F}_{\lambda}\right)} && {H_{s}\left(\mathcal{F}_{\lambda+2\theta||W||_{h}}\right)} \\
	\\
	& {\operatorname{Im}\left(\rho_{\lambda+\theta||W||_{h}}\right)}
	\arrow[from=1-1, to=1-3,"v_{\lambda}^{\lambda+2\theta||W||_{h}}"]
	\arrow[from=1-1, to=3-2,"\overline{\psi}_{\lambda}"]
	\arrow[from=3-2, to=1-3,"\overline{\phi}_{\lambda+\theta||W||_{h}}"]
\end{tikzcd}
\end{equation}
\begin{itemize}
\item \textbf{Diagram \ref{diagram2}} : We can rewrite the diagram as  (unspecified maps are simply induced by set inclusion),
    \begin{equation*}
        \begin{tikzcd}
	{\operatorname{Im}\left(\rho_{\lambda}\right)} && {\operatorname{Im}\left(\rho_{\lambda^{'}}\right)} \\
	{H_{s}\left(\overline{B}_{2}\left(\mathcal{F}_{\lambda+\theta||W||_{h}},k_{2}h\right)\right)} && {H_{s}\left(\overline{B}_{2}\left(\mathcal{F}_{\lambda^{'}+\theta||W||_{h}},k_{2}h\right)\right)} \\
	{H_{s}\left(\mathcal{F}_{\lambda+\theta||W||_{h}}\right)} && {H_{s}\left(\mathcal{F}_{\lambda^{'}+\theta||W||_{h}}\right)}
	\arrow[ from=2-1, to=2-3]
	\arrow["F^{*}_{\lambda^{'}+\theta||W||_{h},k_{2}h}", from=2-3, to=3-3]
	\arrow[ from=1-3, to=2-3]
	\arrow[ from=3-3, to=2-3]
	\arrow[ from=1-1, to=2-1]
	\arrow[ from=1-1, to=1-3]
	\arrow[ from=3-1, to=3-3]
	\arrow["F^{*}_{\lambda+\theta||W||_{h},k_{2}h}",from=2-1, to=3-1]
	\arrow[ from=3-1, to=2-1]
\end{tikzcd}
\end{equation*}
By inclusions the upper face commutes, and the lower face commutes, as for all $\lambda\in\mathbb{R}$, $F^{*}_{\lambda+\theta||W||_{h},k_{2}h}$ comes from a deformation retract. Thus, the diagram is commutative.
\item \textbf{Diagram \ref{diagram1}}: One can check that it can be decomposed as diagram \ref{diagram2} and thus the same reasoning applies.
\item \textbf{Diagram \ref{diagram3}}: Let $[C]\in \operatorname{Im}\left(\rho_{\lambda}\right)$. By construction, we can suppose $C\in C_{s}(\widehat{\mathcal{F}}_{\lambda})$. The map $\overline{\phi}_{\lambda}$ maps $[C]$ to $[C^{'}]$ with,
$$C^{'}=F^{\#}_{\lambda+\theta||W||_{h},k_{2}h}(C,1).$$
Let denote $\overline{F}_{\lambda+\theta||W||_{h},k_{2}h}$ the restriction of $F_{\lambda+\theta||W||_{h},k_{2}h}$ to $F_{\lambda+\theta||W||_{h},k_{2}h}(\widehat{\mathcal{F}}_{\lambda},[0,1])$. It is a deformation retraction from $F_{\lambda+\theta||W||_{h},k_{2}h}(\widehat{\mathcal{F}}_{\lambda},[0,1])$ onto $F_{\lambda+\theta||W||_{h},k_{2}h}(\widehat{\mathcal{F}}_{\lambda},1)$. Thus, by homotopy invariance of singular homology,
\begin{align*}
[C]&=\left[\overline{F}^{\#}_{\lambda+\theta||W||_{h},k_{2}h}(C,1)\right]\\
&=\left[F^{\#}_{\lambda+\theta||W||_{h},k_{2}h}(C,1)\right]\\
&=[C^{'}]\in H_{s}\left(F_{\lambda+\theta||W||_{h},k_{2}h}(\widehat{\mathcal{F}}_{\lambda},[0,1])\right).
\end{align*}
From the last assertion of Proposition \ref{lemmaFiltEquiv} and the upper inclusion of Proposition \ref{inclusion}, we know that,
$$F_{\lambda+\theta||W||_{h},k_{2}h}\left(\widehat{\mathcal{F}}_{\lambda},[0,1]\right)\subset \widehat{\mathcal{F}}_{\lambda+2\theta||W||_{h}}^{\lceil r_{2}\rceil h}$$
thus,
$$[C]=[C^{'}]\in H_{s}\left[\widehat{\mathcal{F}}_{\lambda+2\theta||W||_{h}}^{\lceil r_{2}\rceil h}\right].$$
As $\overline{\psi}_{\lambda+\theta||W||_{h}}$ is simply an inclusion map, it maps $[C^{'}]$ to  $[C^{'}]$, and hence,
$$[C]=[C^{'}]=\overline{\psi}_{\lambda+\theta||W||_{h}}\left(\overline{\phi}_{\lambda}\left([C]\right)\right)\in  H_{s}\left[\widehat{\mathcal{F}}_{\lambda+2\theta||W||_{h}}^{\lceil r_{2}\rceil h}\right]$$
which proves the commutativity.
\item \textbf{Diagram \ref{diagram4}}:
Let $C\in C_{s}\left(\mathcal{F}_{\lambda}\right)$, $\overline{\psi}_{\lambda}([C])=[C]$ as it is simply an inclusion induced map and 
$$\overline{\phi}_{\lambda+\theta||W||_{h}}([C])=\left[F_{\lambda+\theta||W||_{h}, k_{2}h}^{\#}(C,1)\right]=[C]$$
which establish the commutativity.
\end{itemize}
The commutativity of diagrams \ref{diagram2},\ref{diagram1},\ref{diagram3} and \ref{diagram4} means that $\widehat{\mathbb{V}}_{f,s}$ and $\mathbb{V}_{f,s}$ are $2\theta||W||_{h}$ interleaved, and thus we get from the algebraic stability theorem \citep{Chazal2009} that,
\begin{equation*}
d_{b}\left(\operatorname{dgm}\left(\widehat{\mathbb{V}}_{f,s}\right),\operatorname{dgm}\left(\mathbb{V}_{f,s}\right)\right)\leq 2\theta||W||_{h} 
\end{equation*}
and as it holds for all $s\in\mathbb{N}$ and $f\in S_{d}(\mu,R_{\mu})$,
$$\sup\limits_{f\in S_{d}(\mu,R_{\mu})}d_{b}\left(\widehat{\operatorname{dgm}(f)},\operatorname{dgm}(f)\right)\leq2\theta||W||_{h}.$$
We then conclude, using the concentration of $||W||_{h}$. As $W(H)/h^{d/2}$ is a standard Gaussian, we have,
\begin{align*}
\mathbb{P}\left(||W||_{h}\leq t \right)&=\mathbb{P}\left(\bigcup_{H\in C_h}\left\{|W(H)/h^{d}|\leq t \right\}\right)\\
&\leq |C_{h}|\mathbb{P}\left(|W(H)/h^{d/2}|\leq th^{d/2}\right)\\
&\leq 2\left(\frac{1}{h}\right)^{d}\exp\left(-\frac{h^{d}t^{2}}{2}\right)\\
&:= C_{0}\exp\left(-C_{1}t^{2}\right).
\end{align*}
With $C_{1}$ and $C_{2}$ constants depending on $\mu$, $R_{\mu}$ and $d$. Then, 
\begin{align*}
&\mathbb{P}\left(\sup\limits_{f\in S_{d}(\mu,R_{\mu})}d_{b}\left(\widehat{\operatorname{dgm}(f)},\operatorname{dgm}(f)\right)\geq \theta t\right)\\
&\leq \mathbb{P}\left(||W||_{h}\geq t/2\right)\\
&\leq C_{0}\exp\left(-C_{1}t^{2}/4\right)
\end{align*}
and the result follows.
\end{proof}
\noindent We can derive from this result bounds in expectation.
\begin{crl}
\label{estimation borne sup}
Let $p\geq 1$, 
$$\underset{f\in S_{d}(\mu,R_{\mu})}{\sup}\quad\mathbb{E}\left(d_{b}\left(\widehat{\operatorname{dgm}(f)},\operatorname{dgm}(f)\right)^{p}\right)\lesssim \theta^{p}$$
\end{crl}
\begin{proof}
 The sub-Gaussian concentration provided by Theorem \ref{MainProp}, gives that, for all $t>0$,
$$\mathbb{P}\left(\frac{d_{b}\left(\widehat{\operatorname{dgm}(f)},\operatorname{dgm}(f)\right)}{\theta}\geq t\right)\leq \Tilde{C}_{0}\exp\left(-\Tilde{C}_{1}t^{2}\right).$$
 Now, we have,
 \begin{align*}
 &\quad\mathbb{E}\left(\frac{d_{b}\left(\widehat{\operatorname{dgm}(f)},\operatorname{dgm}(f)\right)^{p}}{\theta^{p}}\right)\\
&=\int_{0}^{+\infty}\mathbb{P}\left(\frac{d_{b}\left(\widehat{\operatorname{dgm}(f)},\operatorname{dgm}(f)\right)^{p}}{\theta^{p}}\geq t\right)dt\\
     &\leq \int_{0}^{+\infty}\Tilde{C}_{0}\exp\left(-\Tilde{C}_{1}t^{\sfrac{2}{p}}\right)dt<+\infty.
\end{align*}
\end{proof}
\subsection{Proof of Proposition \ref{inclusion}}
\label{proof inclusion}
This section is dedicated to the proof of Proposition \ref{inclusion}. It relies on Lemma \ref{lmm SSDO7} (stated in Section \ref{Background section}) and on the following lemma.
\begin{lmm}
\label{lmm1}
 Let $f:[0,1]^{d}\rightarrow \mathbb{R}$ and $h>0$. Let $H\subset \mathcal{F}_{\lambda+\theta||W||_{h}}^{c}\cap C_{h}$ and $H^{'}\subset \mathcal{F}_{\lambda-\theta||W||_{h}}\cap C_{h}$.
We then have that,
$$\int_{H}dX-\int_{H}\lambda>0 \text{ and } \int_{H^{'}}dX-\int_{H}\lambda\leq 0.$$
\end{lmm}
\begin{proof}
Let first consider the case where in $H^{\prime} \subset \mathcal{F}_{\lambda-\theta\|W\|_{h}}\cap C_{h}$. We have,
$$
\begin{aligned}
& \int_{H^{\prime}} d X-\int_{H^{\prime}} \lambda \\
& =\int_{H^{\prime}}(f-\lambda)+\theta \int_{H^{\prime}} d W \\
& \leq-\theta\|W\|_{h}h^{d}+\theta\|W\|_{h}h^{d}= 0.
\end{aligned}
$$
Now, let  $H\subset \mathcal{F}_{\lambda+\theta||W||_{h}}^{c}\cap C_{h}$, we have,
$$
\begin{aligned}
& \int_{H} d X-\int_{H} \lambda \\
& =\int_{H}(f-\lambda)+\theta \int_{H} d W \\
& >\theta\|W\|_{h}h^{d}-\theta\|W\|_{h}h^{d}= 0.
\end{aligned}
$$
\end{proof}
\begin{proof}[Proof of Proposition \ref{inclusion}]
We begin by proving the lower inclusion, let $x\in \mathcal{F}_{\lambda-\theta||W||_{h}}$. Without loss of generality, let suppose $x\in M_{i}$, or $x\in \partial M_{i}$ and $\liminf_{z\in M_{i}\rightarrow x}f(z)\leq f(x)$. If,
$$\overline{B}_{2}\left(x,\sqrt{d}h\right)\subset\left(\bigcup\limits_{i=1}^{l}\partial M_{i}\right)^{c}$$
then, $H_{x,h}$, the hypercube of $C_{h,\lambda}$ containing $x$ is included in $\overline{M}_{i}$.  Assumption \textbf{A1} and \textbf{A2} then gives $H_{x,h}\subset \mathcal{F}_{\lambda-\theta||W||_{h}}$. Hence, it follows from Lemma \ref{lmm1} that $H_{x,h}\in C_{h,\lambda}$ and consequently $x\in \widehat{\mathcal{F}}_{\lambda}$. Else, as $\sqrt{d}h/\mu<R_{\mu}$, by Lemma \ref{lmm SSDO7} (which we can apply thanks to Assumption \textbf{A3}), there exists,
$$y\in \left(B_{2}\left(\bigcup\limits_{i=1}^{l}\partial M_{i},{\sqrt{d}h}\right)\right)^{c}\cap M_{i}  \text{ such that } ||x-y||_{2}\leq \sqrt{d}h/\mu.$$
Let $H_{y,h}$ the closed hypercube of $C_{h}$ containing $y$. Hence, $H_{y,h}\subset \overline{M_{i}}.$. Then, Assumption \textbf{A1} and \textbf{A2} ensure that,
$$H_{y,h}\subset \mathcal{F}_{\lambda-\theta||W||_{h}}.$$
Then, Lemma \ref{lmm1} gives $H_{y,h}\in C_{h,\lambda}$ and thus, as $x\in H_{y,h}^{\sqrt{d}h/\mu}$, $x\in \widehat{\mathcal{F}}_{\lambda}$, which proves the lower inclusion.\\\\
For the upper inclusion, let $x\in\left(\mathcal{F}_{\lambda+\theta||W||_{h}}^{\sqrt{d}h}\right)^{c}$, and  $H_{x,h}$ the hypercube of $C_{h}$ containing $x$. We then have, $H_{x,h}\subset \mathcal{F}_{\lambda+\theta||W||_{h}}^{c}$. Hence, Lemma \ref{lmm1} gives that,
$$H_{x,h}\subset \left(\bigcup\limits_{H\in C_{h,\lambda}}H\right)^{c}$$
and thus,
$$\bigcup\limits_{H\in C_{h,\lambda}}H\subset \mathcal{F}_{\lambda+\theta||W||_{h}}^{\sqrt{d}h}.$$
Consequently,
$$\left(\bigcup\limits_{H\in C_{h,\lambda}}H\right)^{\lceil\sqrt{d}/\mu\rceil h}=\widehat{\mathcal{F}}_{\lambda}\subset \mathcal{F}_{\lambda+\theta||W||_{h}}^{(\sqrt{d}+\lceil\sqrt{d}/\mu\rceil)h}$$
and the proof is complete.
\end{proof}
\subsection{Proof of Proposition \ref{lemmaFiltEquiv}}
\label{proof lemmaFiltequiv}
Proposition \ref{lemmaFiltEquiv} is a direct corollary of Theorem 12 from \cite{kim2020homotopy}.
\begin{lmm}{\citep[][Theorem 12]{kim2020homotopy}}
\label{deformation retract}
Let $K\subset[0,1]^{d}$, for all $0\leq r<\operatorname{reach}_{\mu}(K)$,
$\overline{B}_{2}(K,r)$ retracts by deformation onto $K$ and the associated deformation retraction $F: \overline{B}_{2}(K,r)\times[0,1]\rightarrow \overline{B}_{2}(K,r)$ verifies for all $x\in \overline{B}_{2}(K,r)$ and all $t\in[0,1]$, $F(x,t)\in \overline{B}_{2}(x, 2d_{K}(x)/\mu^{2})$.
\end{lmm}
\noindent The first part of the claim is Theorem 12 from \cite{kim2020homotopy} and the second part follows easily from their construction. Elements of proof can be found in Appendix \ref{appendix : def retract}.
\begin{proof}[Proof of Lemma \ref{lemmaFiltEquiv}]
Let $\lambda\in\mathbb{R}$ and denote $I_{\lambda}=\{i\in\{1,...,l\}\text{ s.t. }M_{i}\cap\mathcal{F}_{\lambda}\ne\emptyset\}$. By Assumption \textbf{A1} and \textbf{A2}, we have $\mathcal{F}_{\lambda}=\bigcup_{i\in I_{\lambda}}\overline{M}_{i}$. Thus, by Assumption \textbf{A3}, 
$$\operatorname{reach}_{\mu}\left(\mathcal{F}_{\lambda}\right)\geq R_{\mu}.$$
Hence, by Lemma \ref{deformation retract}, for all $h<R_{\mu}$, there exists a deformation retraction $F_{\lambda,h}:\overline{B}_{2}(\mathcal{F}_{\lambda},h)\times [0,1]\rightarrow \overline{B}_{2}(\mathcal{F}_{\lambda},h)$ from $\overline{B}_{2}(\mathcal{F}_{\lambda},h)$ onto $\mathcal{F}_{\lambda}$ that furthermore verifies, for all $(x,t)\in \overline{B}_{2}(\mathcal{F}_{\lambda},h)\times [0,1]$, $F_{\lambda}(x,t)\in \overline{B}_{2}(x, 2d_{\mathcal{F}_{\lambda}}(x)/\mu^{2})$.
\end{proof}
\section{Discussion}
This work extends the current scope on persistent homology inference from sublevel sets of a noisy signal. We motivate the need to depart from plug-in approaches, highlighting that cubical approximation fails to capture properly the persistent homology of piecewise-constant signals with discontinuities set having a positive $\mu-$reach. To overcome this issue, we propose a method based on image persistence. In the context of the Gaussian white noise model, we show that over the classes $S_{d}(\mu, R_{\mu})$ this approach is consistent and achieves parametric rates.\\\\
This study complements the results obtained in \cite{Henneuse24a}. Although we impose stronger conditions on the signals in the regular region, we significantly relax the conditions on the discontinuities set, allowing us to consider signals with more complex discontinuities, such as those with multiple points or arbitrarily narrow corners. This, once again, underscores the robustness of persistence diagram inference to signal irregularities and highlights the advantages of moving beyond analyses that rely on sup-norm stability.\\\\
As demonstrated in Appendix C of \cite{Henneuse24a}, the method and results presented here can be easily extended to the setting of non-parametric regression. Motivated by applications to mode detection, we are also exploring an extension to the density model, which will be the focus of future work.\\\\
A practical limitation of the proposed method is that it requires some knowledge of the parameters $\mu$ and $R_{\mu}$. This dependence is rather common in geometric inference and TDA. Developing adaptive approaches to handle such dependence on geometric measures is a research topic in its own right.\\\\
Also, from a computational perspective (to perform a numerical evaluation of this method), an important step forward would be to provide an efficient algorithm to compute the image persistence for general filtrations. This would permit to compare this method to plug-in approaches, specifically to the histogram estimator. While \cite{bauer2022efficient} offers an efficient such algorithm, it is tailored for Rips complexes, complicating its direct application to our context. In particular, this would permit to investigate practical selection of the parameters $h$, $r_{1}$ and $r_{2}$.
\section*{Acknowledgements}
The author would like to thank Frédéric Chazal and Pascal Massart for our (many) helpful discussions. The author acknowledge the support of the ANR TopAI chair (ANR–19–CHIA–0001). 
\bibliographystyle{plainnat}
\bibliography{bibliographie}
\appendix
\section{Proof for $q$-tameness}
\label{appendix tameness}
This section is devoted to prove the claim that the persistence diagrams we consider and estimated persistence diagrams we propose are well-defined, by proving that the underlying persistence modules are $q-$tame.
\begin{prp}
\label{LmmQtame1}
Let $f\in S_{d}(\mu,R_{\mu})$ then $f$ is $q$-tame.
\end{prp}
\begin{proof}
 Let $s\in\mathbb{N}$ and $\mathbb{V}_{s,f}$ the persistence module (for the $s-$th homology) associated to the sublevel filtration, $\mathcal{F}$ and for fixed levels $\lambda<\lambda^{'}$ let denote $v_{\lambda}^{\lambda^{'}}$ the associated map. Let $h<R_{\mu}$ and denote $i_{\lambda,h}: H_{s}\left(\mathcal{F}_{\lambda}\right)\rightarrow H_{s}\left(\overline{B}_{2}(\mathcal{F}_{\lambda},h)\right)$. By Lemma \ref{lemmaFiltEquiv}, we have, $$v_{\lambda}^{\lambda^{'}}=v_{\lambda}^{\lambda^{'}}\circ F^{*}_{\lambda,h}\circ i_{\lambda,h}.$$
By assumption \textbf{A1} and $\textbf{A2}$, $\mathcal{F}_{\lambda}$ is compact. As $[0,1]^{d}$ is triangulable,  $\mathcal{F}_{\lambda}$ is covered by finitely many cells of the triangulation, and so there is a finite simplicial complex $K$ such that $\overline{\mathcal{F}_{\lambda}}\subset K \subset \mathcal{F}_{\lambda}^{h}$. Consequently, $i_{\lambda,h}$ factors through the finite dimensional space $H_{s}(K)$ and is then of finite rank by Theorem 1.1 of \cite{Crawley2012}. Thus, $v_{\lambda}^{\lambda^{'}}$ is of finite rank. Hence, $f$ is $q$-tame.
\end{proof}
\begin{prp}
\label{qtamenessEstim}
Let $f:[0,1]^{d}\rightarrow\mathbb{R}$ and $h>0$ then, for all $s\in\mathbb{N}$, $\widehat{\mathbb{V}}_{s,f}$ is $q$-tame. 
\end{prp}
\begin{proof}
The results also follows from by Theorem 1.1 of \cite{Crawley2012}. Let $h>0$ and $\lambda\in \mathbb{R}$. $\widehat{\mathcal{F}}_{\lambda}$ and $\widehat{\mathcal{F}}_{\lambda}^{\lceil r_{2}\rceil h}$ are unions of cubes of $C_{h}$ and thus finite dimensional. Hence, $\operatorname{Im}\left(\rho_{\lambda}\right)$ is finite dimensional. Thus $\widehat{\mathbb{V}}_{s,f}$ is $q$-tame.
\end{proof}
\section{Element of proof for Lemma \ref{deformation retract}}
\label{appendix : def retract}
This section is dedicated to the proof of Lemma \ref{deformation retract}.
\begin{proof}[Proof of Lemma \ref{deformation retract}]
The first part of the claim is Theorem 12 of \cite{kim2020homotopy}. A standard technique to construct deformation retraction in differential topology is to exploit the flow coming from a smooth underlying vector field. In \cite{kim2020homotopy}, their idea is to use the vector field defined on  $\overline{B}_{2}(K,r)\setminus K$, by $W(x)=-\nabla_{K}(x)$. But this vector field is not continuous. To overcome this issue, they construct a locally finite covering $(U_{x_{i}})_{i\in \mathbb{N}}$ of $\overline{B}_{2}(K,r)\setminus K$ and an associated partition of the unity $(\rho_{i})_{{i\in \mathbb{N}}}$, such that $\overline{W}(x)=\sum_{i\in \mathbb{N}}\rho_{i}(x)w(x_{i})$ shares the same dynamic as $W$. More precisely, they show that $\overline{W}$ induces a smooth flow $C$ that can be extended on $\overline{B}_{2}(K,r)\times [0,+\infty[$ such that for all $x\in \overline{B}_{2}(K,r)$, for all $t\geq 2d_{K}(x)/\mu^{2}$, $C(x,t)=C(x,2d_{K}(x)/\mu^{2})\in K$. We make an additional remark, denotes $d_{C}$ the arc length distance along $C$, as $||W||\leq 1$, we have,
\begin{align}
d_{C}(x,C(x,2r/\mu^{2}))&=\int_{0}^{2r/\mu^{2}}\left|\frac{\partial }{\partial t}C(x,t)\right|dt\nonumber\\
&\leq\int_{0}^{2r/\mu^{2}}||W(C(x,t)||_{2}dt\nonumber\\
&\leq 2d_{K}(x)/\mu^{2}\label{arc length bound}
\end{align} Thus for all $t\in[0,+\infty[$, $||x-C(x,t)||_{2}\leq 2r/\mu^{2}$. Now taking, $F(x,s)=C(x,2rs/\mu^{2})$, provide a deformation retract of $\overline{B}_{2}(K,r)$ onto $K$ and the associated retraction $R: x\mapsto F(x,1)$ verifies, $R(x)\in \overline{B}_{2}(x,2d_{K}(x)/\mu^{2})$ by (\ref{arc length bound}).
\end{proof}
\end{document}